\documentclass{amsart}
\usepackage{amsthm}
\usepackage{amsmath}
\usepackage{amssymb}
\usepackage{amsfonts}
\usepackage{latexsym}
\usepackage{enumitem}
\usepackage{mathrsfs}
\usepackage[all]{xy} \SelectTips{eu}{}
\usepackage[hidelinks]{hyperref}

\theoremstyle{plain}
\newtheorem{theorem}{Theorem}
\newtheorem*{theorema}{Theorem A}
\newtheorem*{theoremb}{Theorem B}
\newtheorem{proposition}[theorem]{Proposition}
\newtheorem{corollary}[theorem]{Corollary}
\newtheorem{lemma}[theorem]{Lemma}
\numberwithin{theorem}{section}

\theoremstyle{definition}
\newtheorem{definition}[theorem]{Definition}
\newtheorem{remark}[theorem]{Remark}

\newcommand{\deq}{\:=\:}

\newcommand{\setof}[3][\mspace{1mu}]{\{#1#2 \mid #3#1\}}
\newcommand{\ZZ}{\mathbb{Z}}
\newcommand{\Spec}[1]{\operatorname{Spec}#1}
\newcommand{\Cy}[2][]{\operatorname{Z}^{#1}(#2)}
\newcommand{\lra}{\longrightarrow}
\newcommand{\is}{\cong}
\newcommand{\tp}[3][R]{\nobreak{#2\otimes_{#1}#3}}
\newcommand{\calC}{\mathcal{C}}
\newcommand{\calF}{\mathcal{F}}
\newcommand{\calQ}{\mathcal{Q}}
\newcommand{\calR}{\mathcal{R}}
\newcommand{\calS}{\mathcal{S}}
\newcommand{\calG}{\mathcal{G}}
\newcommand{\calW}{\mathcal{W}}
\newcommand{\calU}{\mathcal{U}}
\newcommand{\catFlatTilde}{\mathsf{C_\textnormal{ac}^\textnormal{Z}}(\FlatX)}
\newcommand{\catCotTilde}{\mathsf{C_\textnormal{ac}^\textnormal{Z}}(\catCot{X})}
\newcommand{\Ch}[1]{\mathsf{C}(#1)}
\newcommand{\Homab}[2]{\operatorname{Hom}(#1,#2)}
\newcommand{\Homqc}[2]{{\mathscr Hom}_{{\rm qc}}(#1,#2)}
\newcommand{\cathom}[3][]{\operatorname{hom}_{#1}(#2,#3)}
\newcommand{\Ext}[4]{\operatorname{Ext}^{#1}_{#2}(#3,#4)}
\newcommand{\Tor}[4]{\operatorname{Tor}_{#1}^{#2}(#3,#4)}
\renewcommand{\H}{\operatorname{H}}
\newcommand{\A}{\mathscr{A}}
\newcommand{\B}{\mathscr{B}}
\newcommand{\E}{\mathscr{E}}
\newcommand{\F}{\mathscr{F}}
\newcommand{\G}{\mathscr{G}}
\newcommand{\K}{\mathscr{K}}
\newcommand{\M}{\mathscr{M}}
\newcommand{\I}{\mathscr{I}}
\newcommand{\N}{\mathscr{N}}
\newcommand{\T}{\mathscr{T}}
\newcommand{\X}{\mathscr{X}}
\newcommand{\C}{\mathscr{C}}
\newcommand{\LL}{\mathscr{L}}
\renewcommand{\P}{\mathscr{P}}
\newcommand{\R}{\mathcal{O}_X}
\newcommand{\Rx}{\mathcal{O}_{X,x}}
\newcommand{\Qcoh}[1]{\mathsf{Qcoh}{(#1)}}
\newcommand{\Ftac}{\textnormal{\bf F}-totally acyclic}
\newcommand{\catK}[1]{\mathsf{K}(#1)}
\newcommand{\catKtac}[1]{\mathsf{K}_\textnormal{tac}(#1)}
\newcommand{\catKFtac}[1]{\mathsf{K}_\textnormal{F-tac}(#1)}
\newcommand{\catKpure}[1]{\mathsf{K}_\textnormal{pac}(#1)}
\newcommand{\catD}[1]{\mathsf{D}(#1)}
\newcommand{\catDFtac}[1]{\mathsf{D}_\textnormal{F-tac}(#1)}
\newcommand{\catGF}[1]{\mathsf{GFlat}({#1})}
\newcommand{\catGFC}[1]{\mathsf{GFC}({#1})}
\newcommand{\catdgCot}[1]{\mathsf{C_\textnormal{semi}}(\catCot{#1})}
\newcommand{\catCot}[1]{\mathsf{Cot}(#1)}
\newcommand{\catFlat}[1]{\mathsf{Flat}(#1)}
\newcommand{\FlatX}{\catFlat{X}}
\newcommand{\catprj}[1]{\mathsf{prj}(#1)}
\newcommand{\catPrj}[1]{\mathsf{Prj}(#1)}
\newcommand{\stabcatGFC}[1]{\mathsf{St}\catGFC{#1}}

\def\urltilda{\kern -.15em\lower .7ex\hbox{\~{}}\kern .04em}
\def\soft#1{\leavevmode\setbox0=\hbox{h}\dimen7=\ht0\advance \dimen7
  by-1ex\relax\if t#1\relax\rlap{\raise.6\dimen7
  \hbox{\kern.3ex\char'47}}#1\relax\else\if T#1\relax
  \rlap{\raise.5\dimen7\hbox{\kern1.3ex\char'47}}#1\relax \else\if
  d#1\relax\rlap{\raise.5\dimen7\hbox{\kern.9ex \char'47}}#1\relax\else\if
  D#1\relax\rlap{\raise.5\dimen7 \hbox{\kern1.4ex\char'47}}#1\relax\else\if
  l#1\relax \rlap{\raise.5\dimen7\hbox{\kern.4ex\char'47}}#1\relax \else\if
  L#1\relax\rlap{\raise.5\dimen7\hbox{\kern.7ex
  \char'47}}#1\relax\else\message{accent \string\soft \space #1 not
  defined!}#1\relax\fi\fi\fi\fi\fi\fi}

\begin{document}

\title[The stable category of Gorenstein flat sheaves]{The stable
  category of Gorenstein flat sheaves on a noetherian scheme}

\author[L.W. Christensen]{Lars Winther Christensen}

\address{L.W.C. \ Texas Tech University, Lubbock, TX 79409, U.S.A.}

\email{lars.w.christensen@ttu.edu}

\urladdr{http://www.math.ttu.edu/\urltilda lchriste}

\author[S. Estrada]{Sergio Estrada}

\address{S.E. \ Universidad de Murcia, Murcia 30100, Spain}

\email{sestrada@um.es}

\urladdr{https://webs.um.es/sestrada/}

\author[P.\ Thompson]{Peder Thompson}

\address{P.T. \ Norwegian University of Science and Technology, 7491
  Trondheim, Norway}

\email{peder.thompson@ntnu.no}

\urladdr{https://folk.ntnu.no/pedertho}

\thanks{L.W.C.\ was partly supported by Simons Foundation
  collaboration grant 428308. S.E.\ was partly supported by grants
  PRX18/00057, MTM2016-77445-P, and 19880/GERM/15 by the Fundaci\'on
  S\'eneca-Agencia de Ciencia y Tecnolog\'{\i}a de la Regi\'on
  de~Murcia and FEDER funds}

\date{14 July 2020}

\keywords{Cotorsion sheaf, Gorenstein flat sheaf, noetherian scheme,
  stable category, totally acyclic complex}

\subjclass[2010]{14F08; 18G35}

\begin{abstract}
  For a semi-separated noetherian scheme, we show that the category of
  cotorsion Gorenstein flat quasi-coherent sheaves is Frobenius and a
  natural non-affine analogue of the category of Gorenstein projective
  modules over a noetherian ring. We show that this coheres perfectly
  with the work of Murfet and Salarian that identifies the pure
  derived category of \Ftac\ complexes of flat quasi-coherent sheaves
  as the natural non-affine analogue of the homotopy category of
  totally acyclic complexes of projective modules.
\end{abstract}

\maketitle

\thispagestyle{empty}

\section*{Introduction}

\noindent
A classic result due to Buchweitz \cite{ROB86} says that the
singularity category of a Gorenstein local ring $A$ is equivalent to
the homotopy category $\catKtac{\catprj{A}}$ of totally acyclic
complexes of finitely generated projective $A$-modules. The latter
category is also equivalent to the stable category of finitely
generated maximal Cohen-Macaulay $A$-modules or, in a different
terminology, to the stable category $\mathsf{StGprj}(A)$ of finitely
generated Gorenstein projective $A$-modules. This second equivalence
extends beyond the realm of Gorenstein local rings and finitely
generated modules: For every ring $A$, the category
$\catKtac{\catPrj{A}}$ of totally acyclic complexes of projective
$A$-modules is equivalent to the stable category $\mathsf{StGPrj}(A)$
of Gorenstein projective $A$-modules. We obtain this folklore result
as a special case of \cite[Corollary~3.9]{CETa}. What is the analogue
in the non-affine setting?

Murfet and Salarian \cite{DMfSSl11} offer a non-affine analogue of the
category $\catKtac{\catPrj{A}}$ over a semi-separated noetherian
scheme $X$ in the form of the Verdier quotient,
\begin{equation*}
  \catDFtac{\FlatX}
  =
  \frac{\catKFtac{\FlatX}}{\catKpure{\FlatX}}\:,
\end{equation*}
of the homotopy category of \Ftac\ complexes of flat quasi-coherent
sheaves on $X$ by its subcategory of pure-acyclic complexes.  Indeed,
for a commutative noetherian ring $A$ of finite Krull dimension and
$X=\Spec(A)$, the categories $\catKtac{\catPrj{A}}$ and
$\catDFtac{\FlatX}$ are equivalent by \cite[Lemma
4.22]{DMfSSl11}. What remains is to identify an analogue of the
category $\mathsf{StGPrj}(A)$ in the non-affine setting, and that is
the goal of this paper.

The stable category of Gorenstein projective modules is a standard
construction that applies to any Frobenius category. The category of
Gorenstein flat modules is rarely Frobenius; it is essentially only
Frobenius if it coincides with the category of Gorenstein projective
modules, see \cite[Theorem 4.5]{CETa}. The cotorsion Gorenstein flat
modules, however, do form a Frobenius category, and a special case of
\cite[Corollary~5.9]{CETa} says that for a commutative noetherian ring
$A$ of finite Krull dimension, the category $\mathsf{StGPrj}(A)$ is
equivalent to the stable category $\stabcatGFC{A}$ of cotorsion
Gorenstein flat modules. This identifies a candidate category and,
indeed, the goal stated above is obtained (in~\ref{cor.main}) with

\begin{theorema}
  Let $X$ be a semi-separated noetherian scheme. The stable category
  $\stabcatGFC{X}$ of cotorsion Gorenstein flat sheaves is equivalent
  to $\catDFtac{\FlatX}$.
\end{theorema}
\noindent In the statement of this theorem, and everywhere else in
this paper, a sheaf means a quasi-coherent sheaf. A sheaf on $X$ is
called cotorsion if it is right Ext-orthogonal to flat sheaves on $X$.

A crucial step towards the equivalence in Theorem~A is to prove (in
\ref{Ftac_thm} and \ref{equalities}) that the sheaves that are both
cotorsion and Gorenstein flat are precisely the sheaves that arise as
cycles in \Ftac\ complexes of flat-cotorsion sheaves. This is exactly
what happens in the affine case, and it transpires that the main
take-away from \cite{CETa} also applies in the non-affine setting: One
should work with sheaves that are both cotorsion and Gorenstein flat
rather than all Gorenstein flat sheaves! One manifestation is a result
(\ref{gorX}) that sharpens \cite[Theorem~4.27]{DMfSSl11}:

\begin{theoremb}
  A semi-separated noetherian scheme $X$ is Gorenstein if and only if
  every acyclic complex of flat-cotorsion sheaves on $X$ is \Ftac.
\end{theoremb}

A second manifestation---actually the result behind Theorem A---is
that the category $\catDFtac{\FlatX}$ considered by Murfet and
Salarian is equivalent to the homotopy category
$\catKFtac{\FlatX \cap \catCot{X}}$ of \Ftac\ complexes of flat
cotorsion sheaves (see \ref{main_thm}). That is, passing from the
affine to the non-affine setting, one can replace the homotopy
category $\catKtac{\catPrj{A}}$ by another homotopy category.

Up to equivalence, the category $\catKFtac{\FlatX \cap \catCot{X}}$
arises in a related, yet different, context. The category of
Gorenstein flat sheaves on a semi-separated noetherian scheme $X$ is
part of a complete hereditary cotorsion pair (see
\ref{theorem.gfcotpair}), and one that is comparable to the cotorsion
pair of flat sheaves and cotorsion sheaves on $X$. Through work of
Hovey~\cite{MHv02} and Gillespie~\cite{JGl15}, these cotorsion pairs
induce a model structure on the category of sheaves on $X$. We prove
(see \ref{stabcats}) that the associated homotopy category is
equivalent to $\catKFtac{\FlatX \cap \catCot{X}}$.


\section{Gorenstein flat sheaves}
\label{sec:Ntot}

\noindent
In this paper, the symbol $X$ denotes a scheme with structure sheaf
$\R$. By a sheaf on $X$ we shall always mean a quasi-coherent sheaf,
and $\Qcoh{X}$ denotes the category of (quasi-coherent) sheaves on
$X$. We frequently add the assumption that $X$ is
\emph{semi-separated}, by which we mean that $X$ has an open affine
covering $\calU$ such that $U \cap V$ is affine for all
$U,V \in \calU$; such a covering is referred to as
\emph{semi-separating}. We use standard cohomological notation for
cochain complexes.

In this first section we show that over a semi-separated noetherian
scheme, one can equivalently define Gorenstein flatness of sheaves
globally, locally, or stalkwise.  Let $A$ be a commutative ring. An
acyclic complex $F$ of flat $A$-modules is called \emph{\Ftac} if the
complex $\tp[A]{I}{F}$ is acyclic for every injective
$A$-module~$I$. An $A$-module $M$ is \emph{Gorenstein flat} if there
exists an \Ftac\ complex $F$ with $M = \Cy[0]{F}$. Denote by
$\catGF{A}$ the category of Gorenstein flat $A$-modules.

\begin{remark}
  \label{rmk:Ftac}
  For an acyclic complex $\F$ of flat sheaves on $X$ there is a
  global, a local, and a stalkwise notion of \textbf{F}-total
  acyclicity:
  \begin{itemize}
  \item For every injective sheaf $\I$ on $X$ the complex
    $\tp[]{\I}{\F}$ is acyclic.
  \item For every open affine $U\subseteq X$ the $\R(U)$-complex
    $\F(U)$ is \Ftac.
  \item For every $x\in X$ the $\Rx$-complex $\F_x$ is \Ftac.
  \end{itemize}
  It is proved in \cite[Lemmas~4.4 and 4.5]{DMfSSl11} that all three
  notions agree if the scheme $X$ is semi-separated
  noetherian. Christensen, Estrada, and Iacob
  \cite[Corollary~2.8]{CEI-17} show that the local notion is
  Zariski-local, and by \cite[Proposition~2.10]{CEI-17} the local and
  global notions agree if $X$ is semi-separated and quasi-compact
  (which is weaker than noetherian).
\end{remark}

\begin{definition}
  \label{dfn:GF}
  Assume that $X$ is semi-separated noetherian. An acyclic complex
  $\F$ of flat sheaves on $X$ is called \textbf{F}-\emph{totally
    acyclic} if it satisfies the equivalent conditions in
  Remark~\ref{rmk:Ftac}. A sheaf $\M$ on $X$ is called
  \emph{Gorenstein flat} if there exists an \textbf{F}-totally acyclic
  complex $\F$ of flat sheaves on $X$ with $\M = \Cy[0]{\F}$. Denote
  by $\catGF{X}$ the category of Gorenstein flat sheaves on $X$.
\end{definition}

Over any scheme, Gorenstein flatness can also be defined locally or
stalkwise, and we proceed to show that these notions agree with
Gorenstein flatness as defined above if the scheme is semi-separated
noetherian.

\begin{definition}
  \label{def:G.flat.def}
  A sheaf $\M$ on $X$ is called \emph{locally Gorenstein flat} if for
  every open affine subset $U\subseteq X$ the $\R(U)$-module $\M(U)$
  is Gorenstein flat, and $\M$ is called \emph{stalkwise Gorenstein
    flat} if $\M_x$ is a Gorenstein flat $\Rx$-module for every
  $x\in X$.
\end{definition}

Like local \textbf{F}-total acyclicity, local Gorenstein flatness is a
Zariski-local property, at least under mild assumptions on the
scheme. As shown in \cite{CEI-17}, this follows from the next
proposition.

\begin{proposition}
  \label{prp:G.flat.local}
  Let $\varphi\colon A\to B$ be a flat homomorphism of commutative
  rings.
  \begin{enumerate}
  \item[\rm (a)] If $M$ is a Gorenstein flat $A$-module, then
    $\tp[A]{B}{M}$ is a Gorenstein flat $B$-module.
  \item[\rm (b)] Assume that $A$ is coherent and $\varphi$ is
    faithfully flat.  An $A$-module $M$ is Gorenstein flat if the
    $B$-module $\tp[A]{B}{M}$ is Gorenstein flat.
  \end{enumerate}
\end{proposition}

\begin{proof}
  (a) Let $F$ be an \Ftac\ complex of flat $A$-modules with
  $M = \Cy[0]{F}$. By \cite[Proposition 2.7(1)]{CEI-17} the
  $B$-complex $\tp[A]{B}{F}$ is an \Ftac\ complex of flat $B$-modules,
  so $\tp[A]{B}{M} = \Cy[0]{\tp[A]{B}{F}}$ is a Gorenstein flat
  $B$-module.

  (b) It follows from work of \v{S}aroch and \v{S}{\soft{t}}ov{\'{\i}}{\v{c}}ek
  \cite[Corollary~4.12]{JSrJSt20} that the category $\catGF{A}$ is
  closed under extensions, so the assertion is immediate from a result
  of Christensen, K\"oksal, and Liang \cite[Theorem 1.1]{CKL-17}.
\end{proof}

\begin{corollary}
  \label{cor:G.flat.local}
  Assume that $X$ is locally coherent. A sheaf $\M$ on $X$ is locally
  Gorenstein flat if there exists an open affine covering $\calU$ of
  $X$ such that the $\R(U)$-module $\M(U)$ is Gorenstein flat for
  every $U\in \calU$.
\end{corollary}

\begin{proof}
  Proposition \ref{prp:G.flat.local} shows that Gorenstein flatness is
  an ascent--descent property for modules over commutative coherent
  rings. Now invoke \cite[Lemma~2.4]{CEI-17}.
\end{proof}

\begin{theorem}
  \label{thm:equiv_Noetss}
  Assume that $X$ is semi-separated noetherian. For a sheaf $\M$ on
  $X$ the following conditions are equivalent.
  \begin{enumerate}
  \item[$(i)$] $\M$ is Gorenstein flat.
  \item[$(ii)$] $\M$ is locally Gorenstein flat.
  \item[$(iii)$] $\M$ is stalkwise Gorenstein flat.
  \end{enumerate}
\end{theorem}

\begin{proof}
  The implication $(i) \Rightarrow (ii)$ is trivial by the definition
  of Gorenstein flatness; see Remark~\ref{rmk:Ftac}.

  $(ii)\Rightarrow (iii)$: Let $x\in X$ and choose an open affine
  subset $U \subseteq X$ with $x\in U$. Localization is exact and
  commutes with tensor products, so it preserves Gorenstein flatness
  of modules, whence the module $\M_x \is \M(U)_x$ is Gorenstein flat
  over the local ring $\Rx \is \R(U)_x$.

  $(iii)\Rightarrow (i)$: This argument is inspired by Yang and Liu
  \cite[Lemmas 3.8 and 3.9]{GYnZLi12}.  Let $\calU=\{U_0,\ldots,U_n\}$
  be a semi-separating open affine covering of $X$. For every $x\in X$
  there is a short exact sequence of $\Rx$-modules,
  \begin{equation}
    \label{eq:a1}
    0 \lra \M_x \lra F_x \lra T_x \lra 0\:,
  \end{equation}
  where $F_x$ is flat and $T_x$ is Gorenstein flat.  For $x\in X$ and
  $U\in \calU$ consider the canonical maps
  \begin{equation*}
    i_x\colon \Spec(\Rx)\longrightarrow X \quad\text{and}\quad
    i_U\colon \Spec(\R(U))\longrightarrow X\:;
  \end{equation*}
  for $x\in U$ the map $i_x$ factors through $i_U$.  The map
  $\M \to \prod_{x\in X} (i_x)_*(\widetilde{\M_x})$ is a monomorphism
  locally at every $y\in X$, as one has
  \begin{equation*}
    \prod_{x\in X} (i_x)_*(\widetilde{\M_x}) \is
    (i_y)_*(\widetilde{\M_y}) \oplus \prod_{x\in X\setminus \{y\}}
    (i_x)_*(\widetilde{\M_x})\:,
  \end{equation*}
  so it is a monomorphism in $\Qcoh{X}$.  Now \eqref{eq:a1} yields a
  monomorphism
  \begin{equation*}
    \M \lra \prod_{x \in X} (i_x)_*(\widetilde{F_x})\:,
  \end{equation*}
  so with $\F^0= \prod_{x\in X} (i_x)_*(\widetilde{F_x})$ there is an
  exact sequence in $\Qcoh{X}$
  \begin{equation}
    \label{eq:a2}
    0 \lra \M \lra \F^0 \lra \K^1\lra 0\:.
  \end{equation}

  The first goal is to show that $\F^0$ is flat.  For every $x\in X$
  choose only one element $U_k$ in $\calU$ with $x\in U_k$, and for
  $k=0,\ldots,n$ let $I_k\subseteq U_k$ denote the corresponding
  subset such that $X$ is the disjoint union $\bigcup_{k=0}^n
  I_k$. Now one has
  \begin{equation*}
    \F^0=\prod_{x \in X} (i_x)_*(\widetilde{F_x})
    =\bigoplus_{k=0}^n\prod_{x\in I_k}(i_{U_k})_*(\widetilde{F_x})
    \cong $$ $$\stackrel{(\dag)}{\cong}\bigoplus_{k=0}^n (i_{U_k})_*(\prod_{x\in I_k}\widetilde{F_x})\cong \bigoplus_{k=0}^n (i_{U_k})_*(\widetilde{\prod_{x\in I_k} F_x})\:,
  \end{equation*}
  where the isomorphism $(\dag)$ holds as $(i_{U_k})_*$, being a right
  adjoint functor, preserves direct products.  Since $F_x$ is a flat
  $\R(U_k)$-module, and $ \R(U_k)$ is noetherian, it follows that
  $\prod_{x\in I_k}F_x$ is a flat $\R(U_k)$-module.  Hence $\F^0$ is a
  flat sheaf.

  The second goal is to show that $\K^1$ is Gorenstein flat locally at
  every point $y\in X$. Consider the commutative diagram of
  $\mathcal{O}_{X,y}$-modules
  \begin{equation*}
    \xymatrix{
      0 \ar[r] & \M_y \ar[r] \ar@{=}[d]& \F^0_y \ar[r]
      \ar[d]^-\pi & \K^1_y \ar[r] \ar[d]^-\varpi& 0 \\
      0 \ar[r] & \M_y \ar[r] & F_y \ar[r] & T_y \ar[r] & 0 \\
    }
  \end{equation*}
  where $\pi$ is the canonical projection with kernel $L$; this is a
  flat module as it is the kernel of an epimorphism between flat
  $\Rx$-modules. By the Snake Lemma $\varpi$ is surjective with kernel
  $L$, so $\K^1_y$ is Gorenstein flat; see e.g.\
  \cite[Corollary~4.12]{JSrJSt20}.

  Let $\I$ be an injective sheaf on $X$; we argue that \eqref{eq:a2}
  remains exact after tensoring with $\I$ by showing that
  $\Tor{1}{\Qcoh{X}}{\I}{\K^1}=0$ holds. For $x\in X$ let $J(x)$ be
  the sheaf on $\Spec(\Rx)$ associated to the injective hull of the
  residue field of the local ring $\Rx$. One has
  \begin{equation*}
    \I \is \bigoplus_{x\in X}(i_x)_*J(x)^{(\Lambda_x)}
  \end{equation*}
  for some index sets $\Lambda_x$; see Hartshorne \cite[Proposition
  II.7.17]{rad}. Therefore, it suffices to verify that
  $\Tor{1}{\Qcoh{X}}{(i_x)_*J(x)}{\K^1}=0$ holds for every $x\in
  X$. This can be verified locally, and every localization
  $\Tor{1}{\Qcoh{X}}{(i_x)_*J(x)}{\K^1}_{x'}$ is $0$ or isomorphic to
  $\Tor{1}{\Rx}{J(x)}{\K^1_x}$, and the latter is also $0$ as $\K^1_x$
  is a Gorenstein flat $\Rx$-module.

  Repeating this process, one gets an exact sequence of sheaves
  \begin{equation}
    \label{eq:a3}
    0 \lra \M \lra \F^0 \lra \F^1 \lra \F^2 \lra \cdots
  \end{equation}
  which remains exact after tensoring with any injective sheaf on
  $X$. Since $X$, in particular, is semi-separated quasi-compact,
  every sheaf is a homomorphic image of a flat sheaf; see for example
  Efimov and Positselski \cite[Lemma~A.1]{AIELPs15}.  Therefore, there
  is an exact sequence
  \begin{equation}
    \label{eq:a4}
    0\lra \K^{-1}\lra \F^{-1} \lra \M \lra 0
  \end{equation}
  with $\F^{-1}$ a flat sheaf.  The class of Gorenstein flat modules
  is closed under kernels of epimorphisms, see e.g.\
  \cite[Corollary~4.12]{JSrJSt20}, so $\K^{-1}_x$ is a Gorenstein flat
  $\Rx$-module for every $x \in X$. By the same argument as above the
  sequence \eqref{eq:a4} remains exact after tensoring with any
  injective sheaf on $X$. Repeating this process, one obtains an exact
  sequence
  \begin{equation}
    \label{eq:a5}
    \cdots \lra \F^{-3} \lra \F^{-2} \lra \F^{-1} \lra \M \lra 0
  \end{equation}
  that remains exact after tensoring with any injective sheaf on
  $X$. Splicing together \eqref{eq:a3} and \eqref{eq:a5} one gets per
  Definition \ref{dfn:GF} an \Ftac\ complex of flat sheaves,
  $\F = \cdots \to \F^{-1} \to \F^0 \to \F^1 \to \cdots$. Thus,
  $\M = \Cy[0]{\F}$ is Gorenstein flat.
\end{proof}

Henceforth we work mainly over semi-separated noetherian schemes. In
that setting we consistently refer to the sheaves described in
Theorem~\ref{thm:equiv_Noetss} by their shortest name: Gorenstein
flat; some proofs, though, rely crucially on their local properties.


\section{The Gorenstein flat model structure on $\Qcoh{X}$}

\noindent
Let $\calG$ be a Grothendieck category, that is, an abelian category
that has colimits, exact direct limits (filtered colimits), and a
generator.  A class $\calC$ of objects in $\calG$ is called
\emph{resolving} if it contains all projective objects and is closed
under extensions and kernels of epimorphisms. To a class $\calC$ of
objects in $\calG$ one associates the orthogonal classes
\begin{align*}
  \calC^\perp &\deq \setof{G\in \calG}
                {\Ext{1}{\calG}{C}{G}=0\ \text{for all } C\in \calC}\ \text{ and}\\
  ^{\perp}\calC &\deq \setof{G\in \calG} {\Ext{1}{\calG}{G}{C}=0\
                  \text{for all } C\in \calC}\:.
\end{align*}
Let $\calS \subseteq \calC$ be a set. The pair $(\calC,\calC^{\perp})$
is said to be \emph{generated by the set $\calS$} if an object $G$
belongs to $\calC^\perp $ if and only if $\Ext{1}{\calG}{C}{G}=0$
holds for all $C\in \calS$.  A pair $(\calF,\calC)$ of classes in
$\calG$ with $\calF^{\perp}=\calC$ and $^{\perp}\calC=\calF$ is called
a \emph{cotorsion pair}. The intersection $\calF \cap \calC$ is called
the \emph{core} of the cotorsion pair.

A cotorsion pair $(\calF,\calC)$ in $\calG$ is called
\emph{hereditary} if for all $F \in \calF$ and $C \in \calC$ one has
$\Ext{i}{\calG}{F}{C} = 0$ for all $i \ge 1$. Notice that the class
$\calF$ in this case is resolving.

A cotorsion pair $(\calF,\calC)$ in $\calG$ is called \emph{complete}
provided that for every $G\in \calG$ there are short exact sequences
$0\to C\to F\to G \to 0$ and $0\to G\to C'\to F'\to 0$ with
$F,F'\in \calF$ and $C,C'\in \calC$.

  \subsection*{Abelian model category structures from cotorsion pairs}
  Gillespie \cite{JGl15} shows how to construct a hereditary abelian
  model structure on $\calG$ from two comparable cotorsion
  pairs. Namely, if $(\calQ,\widetilde{\calR})$ and
  $(\widetilde{\calQ},\calR)$ are complete hereditary cotorsion pairs
  in $\calG$ with $\widetilde{\calR}\subseteq \calR$,
  $\widetilde{\calQ}\subseteq \calQ$, and
  $\calQ\cap\widetilde{\calR}=\widetilde{\calQ}\cap\calR$, then there
  exists a unique thick (i.e.\ full, closed under direct summands, and
  having the two-out-of-three property) subcategory $\calW$ of $\calG$
  such that $\widetilde{\calQ} = \calQ \cap \calW$ and
  $\widetilde{\calR} = \calR \cap \calW$. In other words
  $(\calQ,\calW,\calR)$ is a so-called \emph{Hovey triple}, and from
  work of Hovey~\cite{MHv02} it is known that there is a unique
  abelian model structure on $\calG$ in which $\calQ$, $\calR$, and
  $\calW$ are the classes of cofibrant, fibrant, and trivial objects,
  respectively; refer to \cite{MHv02} for this standard
  terminology. We are now going to apply this machine to cotorsion
  pairs with $\widetilde{\calQ}$ and $\calQ$ the categories of flat
  and Gorenstein flat sheaves on $X$.

  \begin{remark}
    \label{rmk:locfinrep}
    If $X$ is semi-separated quasi-compact, then $\Qcoh{X}$ is a
    locally finitely presentable Grothendieck category. This was
    proved already in EGA \cite[I.6.9.12]{egaIsv}, though not using
    that terminology. Being a Grothendieck category, $\Qcoh{X}$ has a
    generator and hence, by \cite[Lemma~A.1]{AIELPs15}, a flat
    generator. Sl\'avik and {\v{S}}{\soft{t}}ov{\'{\i}}{\v{c}}ek \cite{ASlJSt}
    have recently proved that if $X$ is quasi-separated and
    quasi-compact, then $\Qcoh{X}$ has a flat generator if and only if
    $X$ is semi-separated.
  \end{remark}

  \begin{theorem}
    \label{theorem.gfcotpair}
    Assume that $X$ is semi-separated noetherian. The pair
    \begin{equation*}
      (\catGF{X},\catGF{X}^\perp)
    \end{equation*}
    is a complete hereditary cotorsion pair.
  \end{theorem}

\begin{proof}
  For an open affine subset $U \subseteq X$ we write $\catGF{U}$ for
  the category $\catGF{\R(U)}$ of Gorenstein flat $\R(U)$-modules.
  For every open affine subset $U \subseteq X$ the pair
  $(\catGF{U},\catGF{U}^\perp)$ is a complete hereditary cotorsion
  pair; see Enochs, Jenda, and L\'opez-Ramos \cite[Theorems 2.11 and
  2.12]{EJL-04}. The proof of \cite[Theorem~2.11]{EJL-04} shows that
  the pair is generated by a set $\mathcal{S}_U$; see also the more
  precise statement in \cite[Corollary~4.12]{JSrJSt20}.

  A result of Estrada, Guil Asensio, Prest, and Trlifaj
  \cite[Corollary 3.15]{EGPT-12} now shows that
  $(\catGF{X},\catGF{X}^\perp)$ is a complete cotorsion pair. Indeed,
  the flat generator of $\Qcoh{X}$ belongs to $\catGF{X}$. As the
  quiver in \cite[Notation~3.12]{EGPT-12} one takes the quiver with
  vertices all open affine subsets of $X$, and the class $\mathcal{L}$
  in \cite[Corollary 3.15]{EGPT-12} is in this case
  \begin{equation*}
    \mathcal{L} = \{\LL \in \Qcoh{X} \mid \LL(U) \in \mathcal{S}_U
    \text{ for every open affine subset } U \subseteq X \}\:.
  \end{equation*}
  Moreover since $(\catGF{U},\catGF{U}^\perp)$ is hereditary, the
  class $\catGF{U}$ is resolving for every open affine subset
  $U\subseteq X$. It follows that $\catGF{X}$ is also resolving,
  whence $(\catGF{X},\catGF{X}^\perp)$ is hereditary as $\catGF{X}$
  contains a generator; see Saor\'{\i}n and \v{S}{\soft{t}}ov{\'{\i}}{\v{c}}ek
  \cite[Lemma 4.25]{MSrJSt11}.
\end{proof}

Let $X$ be a semi-separated noetherian scheme. By $\FlatX$ we denote
the category of flat sheaves on $X$.  The proof of the next result is
modeled on an argument due to Estrada, Iacob, and P\'erez
\cite[Proposition 4.1]{EIP}.

\begin{lemma}
  \label{lem.compatibility}
  Assume that $X$ is semi-separated noetherian. In $\Qcoh{X}$ one has
  \begin{equation*}
    \catGF{X}\cap\catGF{X}^\perp=\FlatX\cap\FlatX^\perp\:.
  \end{equation*}
\end{lemma}

\begin{proof}
  ``$\subseteq $'': Let $\M\in \catGF{X}\cap \catGF{X}^\perp$.  The
  inclusion $\FlatX\subseteq \catGF{X}$ yields
  $\catGF{X}^\perp\subseteq \FlatX^\perp$, so it remains to show that
  $\M$ is flat. Since $\M$ is in $\catGF{X}$ there is an exact
  sequence in $\Qcoh{X}$,
  \begin{equation*}
    0 \lra \M \lra \F \lra \N \lra 0\:,
  \end{equation*}
  with $\F$ a flat sheaf and $\N$ a Gorenstein flat sheaf on
  $X$. Since $\M$ belongs to $\catGF{X}^\perp$ the sequence splits,
  whence $\M$ is flat.

  ``$\supseteq $'': Let $\M\in \FlatX\cap \FlatX^\perp$. As the
  inclusion $\FlatX \subseteq \catGF{X}$ holds, it remains to show
  that $\M$ is in $\catGF{X}^\perp$. Since
  $(\catGF{X},\catGF{X}^\perp)$ is a complete cotorsion pair in
  $\Qcoh{X}$, see Theorem \ref{theorem.gfcotpair}, there is an exact
  sequence in $\Qcoh{X}$,
  \begin{equation}\label{eq5}
    0\lra \M\lra \G\lra \N\lra 0 \:,
  \end{equation}
  with $\G\in \catGF{X}^\perp$ and $\N\in \catGF{X}$.  Moreover, since
  $\catGF{X}$ is closed under extensions by Theorem
  \ref{theorem.gfcotpair}, also $\G$ belongs to $\catGF{X}$. Thus the
  sheaf $\G$ is in $\catGF{X}\cap \catGF{X}^\perp$, so by the
  containment already proved $\G$ is flat. Since $\M$ is also flat, it
  follows that
  \begin{equation*}
    \operatorname{flat\,dim}_{\R(U)}\N(U)\leq 1
  \end{equation*}
  holds for every open affine subset $U \subseteq X$. Thus $\N(U)$ is
  a Gorenstein flat $\R(U)$-module of finite flat dimension and,
  therefore, flat; see \cite[Corollary 10.3.4]{rha}. It follows that
  $\N$ is a flat sheaf. Since $\M\in \FlatX^\perp$ by assumption, the
  sequence (\ref{eq5}) splits. Therefore, $\M$ is a direct summand of
  $\G$ and thus in $\catGF{X}^\perp$.
\end{proof}

We call sheaves in the subcategory $\catCot{X} = \FlatX^\perp$ of
$\Qcoh{X}$ \emph{cotorsion.} Sheaves in the intersection
$\FlatX \cap \catCot{X}$ are called \emph{flat-cotorsion.}

\begin{remark}
  \label{rmk:cotpair}
  Assume that $X$ is semi-separated quasi-compact. In this case the
  category $\FlatX$ contains a generator for $\Qcoh{X}$, so it follows
  from work of Enochs and Estrada \cite[Corollary~4.2]{EEnSEs05} that
  $(\FlatX,\catCot{X})$ is a complete cotorsion pair, and since
  $\FlatX$ is resolving it follows from \cite[Lemma 4.25]{MSrJSt11}
  that the pair $(\FlatX,\catCot{X})$ is hereditary. This fact can
  also be deduced from work of Gillespie \cite[Proposition~6.4]{JGl07}
  and Hovey \cite[Corollary~6.6]{MHv02}.
\end{remark}

The next theorem establishes what we call the \emph{Gorenstein flat
  model structure} on $\Qcoh{X}$; it may be regarded as a non-affine
version of \cite[Theorem 3.3]{JGl17}.

\begin{theorem}
  \label{theorem.modelGflat}
  Assume that $X$ is semi-separated noetherian. There exists a unique
  abelian model structure on $\Qcoh{X}$ with $\catGF{X}$ the class of
  cofibrant objects and $\catCot X$ the class of fibrant objects. In
  this structure $\FlatX$ is the class of trivially cofibrant objects
  and $\catGF{X}^\perp$ is the class of trivially fibrant objects.
\end{theorem}

\begin{proof}
  It follows from Theorem~\ref{theorem.gfcotpair} and
  Remark~\ref{rmk:cotpair}, that $(\catGF{X},\catGF{X}^\perp)$ and
  $(\FlatX,\catCot{X})$ are complete hereditary cotorsion pairs. Every
  flat sheaf is Gorenstein flat, and by Lemma~\ref{lem.compatibility}
  the two pairs have the same core, so they satisfy the conditions in
  \cite[Theorem 1.2]{JGl15}. Thus the pairs determine a Hovey triple,
  and by \cite[Theorem~2.2]{MHv02} a unique abelian model category
  structure on $\Qcoh{X}$ with fibrant and cofibrant objects as
  asserted.
\end{proof}

\begin{corollary}
  \label{corollary.modelGflat}
  Assume that the scheme $X$ is semi-separated noetherian. The
  category $\catCot{X} \cap \catGF{X}$ is Frobenius and the
  projective--injective objects are the flat-cotorsion sheaves. Its
  associated stable category is equivalent to the homotopy category of
  the Gorenstein flat model structure.
\end{corollary}

\begin{proof}
  Applied to the Gorenstein flat model structure from the theorem,
  \cite[Proposition 5.2(4)]{JGl11} shows that
  $\catCot{X} \cap \catGF{X}$ is a Frobenius category with the stated
  projective--injective objects. The last assertion follows from
  \cite[Corollary 5.4]{JGl11}.
\end{proof}


\section{Acyclic complexes of cotorsion sheaves}

\noindent
We assume throughout this section that $X$ is semi-separated
quasi-compact. The category of cochain complexes of sheaves on $X$ is
denoted $\Ch{\Qcoh{X}}$. The goal is to establish a result,
Theorem~\ref{acyc.cotor} below, which in the affine case is proved by
Bazzoni, Cort\'es Izurdiaga, and Estrada \cite[Theorem 1.3]{BCE}. It
says, in part, that every acyclic complex of cotorsion sheaves has
cotorsion cycles. Our proof is inspired by arguments of Hosseini
\cite{EHs19} and \v{S}{\soft{t}ov{\'{\i}}{\v{c}}ek \cite{JSt}.

Let $\catFlatTilde$ denote the full subcategory of $\Ch{\Qcoh{X}}$
whose objects are the acyclic complexes $\F$ of flat sheaves with
$\Cy[n]{\F}\in \catFlat{X}$ for every $n\in\ZZ$; similarly, let
$\catCotTilde$ denote the full subcategory whose objects are the
acyclic complexes $\C$ of cotorsion sheaves with
$\Cy[n]{\C}\in \catCot{X}$ for every $n\in \ZZ$. Further,
$\catdgCot{X}$ denotes the category of complexes $\C$ of cotorsion
sheaves with the property that the total Hom complex $\Homab{\F}{\C}$
of abelian groups is acyclic for every complex $\F\in
\catFlatTilde$. In the literature such complexes are referred to as
dg- or semi-cotorsion complexes; it is part of
Theorem~\ref{acyc.cotor} that every complex of cotorsion sheaves on
$X$ has this property.

\begin{remark}
  \label{fc}
  The pair $(\catFlatTilde,\catdgCot{X})$ is by
  \cite[Theorem~6.7]{JGl07} and \cite[Theorem~2.2]{MHv02} a complete
  cotorsion pair in $\Ch{\Qcoh{X}}$.
\end{remark}

For complexes $\A$ and $\B$ of sheaves on $X$, let $\Homab{\A}{\B}$
denote the standard total Hom complex of abelian groups. There is an
isomorphism
\begin{equation}
  \label{eq:11}
  \Ext{1}{\Ch{\Qcoh{X}},\mathrm{dw}}{\A}{\Sigma^{n-1}\B}
  \cong \H^n\Homab{\A}{\B}\:,
\end{equation}
where $\Ext{1}{\Ch{\Qcoh{X}},\mathrm{dw}}{\A}{\Sigma^{n-1}\B}$ is the
subgroup of $\Ext{1}{\Ch{\Qcoh{X}}}{\A}{\Sigma^{n-1}\B}$ consisting of
degreewise split short exact sequences; see e.g.\ \cite[Lemma
2.1]{JGl04}. For a complex $\F$ of flat sheaves and a complex $\C$ of
cotorsion sheaves, every extension $0 \to \C \to \X \to \F \to 0$ is
degreewise split, so \eqref{eq:11} reads
\begin{equation}
  \label{eq:11a}
  \Ext{1}{\Ch{\Qcoh{X}}}{\F}{\Sigma^{n-1}\C} \cong \H^n\Homab{\F}{\C}\:.
\end{equation}

\begin{lemma}
  \label{lemma.stov}
  Let $(\F_\lambda)_{\lambda\in \Lambda}$ be a direct system of
  complexes in $\catFlatTilde$. If each complex $\F_\lambda$ is
  contractible, then
  \begin{equation*}
    \Ext{1}{\Ch{\Qcoh{X}}}{\varinjlim_{\lambda \in \Lambda} \F_\lambda}{\C}
    \cong \H^1\Homab{\varinjlim_{\lambda \in \Lambda} \F_\lambda}{\C} = 0
  \end{equation*}
  holds for every complex $\C$ of cotorsion sheaves on $X$.
\end{lemma}

\begin{proof}
  The category $\Qcoh{X}$ is locally finitely presentable and,
  therefore, finitely accessible; see Remark~\ref{rmk:locfinrep}. It
  follows that the results, and arguments, in \cite{JSt} apply.  The
  argument in the proof of \cite[Proposition 5.3]{JSt} yields an exact
  sequence,
  \begin{equation}
    \label{eq1}
    0 \lra \K \lra \bigoplus_{\lambda \in \Lambda} \F_\lambda \lra
    \varinjlim_{\lambda \in \Lambda} \F_\lambda \lra 0\:,
  \end{equation}
  where $\K$ is filtered by finite direct sums of complexes
  $\F_\lambda$. That is, there is an ordinal number $\beta$ and a
  filtration $(\K_\alpha \mid \alpha < \beta)$, where $\K_0=0$,
  $\K_\beta = \K$, and
  $\K_{\alpha+1}/\K_\alpha\cong \bigoplus_{\lambda \in
    J_\alpha}\F_\lambda$ for $\alpha < \beta$ and $J_\alpha$ a finite
  set.

  Let $\C$ be a complex of cotorsion sheaves on $X$. As
  $\catFlatTilde$ is closed under direct limits, one has
  $\varinjlim\F_\lambda \in \catFlatTilde$. Thus,
  $\Ext{1}{\Qcoh{X}}{(\varinjlim\F_\lambda)^i}{\C^{j}}=0$ holds for
  all $i,j\in\mathbb{Z}$, whence there is an exact sequence of
  complexes of abelian groups:
  \begin{equation}
    \label{eq22}
    0 \lra \Homab{\varinjlim_{\lambda\in \Lambda} \F_\lambda}{\C} \lra
    \Homab{\bigoplus_{\lambda\in \Lambda} \F_\lambda}{\C}
    \lra  \Homab{\K}{\C} \lra 0\:.
  \end{equation}
  By \eqref{eq:11a} it now suffices to show that the left-hand complex
  in this sequence is acyclic. The middle complex is acyclic because
  each complex $\F_\lambda$ and, therefore, the direct sum
  $\bigoplus \F_\lambda$ is contractible. Thus it is enough to prove
  that $\Homab{\K}{\C}$ is acyclic. Since $\FlatX$ is resolving, it
  follows from \eqref{eq1} that $\K$ is a complex of flat sheaves.  As
  $\C$ is a complex of cotorsion sheaves, \eqref{eq:11a} yields
  \begin{equation*}
    \H^n\Homab{\K}{\C} \cong
    \Ext{1}{\Ch{\Qcoh{X}}}{\K}{\Sigma^{n-1}\C}\:.
  \end{equation*}
  Hence, it suffices to show that
  $\Ext{1}{\Ch{\Qcoh{X}}}{\K}{\Sigma^{n-1}\C}=0$ holds for all
  $n\in \ZZ$.  Let $(\K_{\alpha}\mid \alpha\leq \lambda)$ be the
  filtration of $\K$ described above. For every $n\in\ZZ$ one has
  \begin{equation*}
    \Ext{1}{\Ch{\Qcoh{X}}}{\bigoplus_{\lambda \in J_{\alpha}}\F_\lambda}{\Sigma^{n-1}\C}=0\:,
  \end{equation*}
  so Eklof's lemma \cite[Proposition 2.10]{JSt} yields
  $\Ext{1}{\Ch{\Qcoh{X}}}{\K}{\Sigma^{n-1}\C}=0$.
\end{proof}

\begin{theorem}
  \label{acyc.cotor}
  Assume that $X$ is semi-separated quasi-compact.  Every complex of
  cotorsion sheaves on $X$ belongs to $\catdgCot{X}$, and every
  acyclic complex of cotorsion sheaves belongs to $\catCotTilde$.
\end{theorem}

\begin{proof}
  As $(\catFlatTilde,\catdgCot{X})$ is a cotorsion pair, see
  Remark~\ref{fc}, the first assertion is that for every complex $\M$
  of cotorsion sheaves and every $\F$ in $\catFlatTilde$ one has
  $\Ext{1}{\Ch{\Qcoh{X}}}{\F}{\M} = 0$. Fix $\F\in\catFlatTilde$ and a
  semi-separating open affine covering $\calU = \{U_0,\ldots,U_d\}$ of
  $X$. Consider the double complex of sheaves obtained by taking the
  \v Cech resolutions of each term in $\F$; see Murfet
  \cite[Section~3.1]{DMf-phd}. The rows of the double complex form a
  sequence in $\Ch{\Qcoh{X}}$:
  \begin{equation}
    \label{eq44}
    0 \lra \F \lra \C^0(\calU,\F) \lra \C^1(\calU,\F) \lra \cdots
    \lra \C^d(\calU,\F) \lra 0
  \end{equation}
  with
  \begin{equation*}
    \C^p(\calU,\F) =
    \bigoplus_{j_0<\cdots<j_p}i_*(\widetilde{\F({U_{j_0,\ldots,j_p}}}))\:,
  \end{equation*}
  where $j_0,\ldots,j_p$ belong to the set $\{0,\ldots,d\}$ and
  $i\colon\, U_{j_0,\ldots,j_p}\lra X$ is the inclusion of the open
  affine subset $U_{j_0,\ldots,j_p} = U_{j_0}\cap\ldots\cap U_{j_p}$
  of $X$. For a tuple of indices $j_0 < \cdots < j_p$, the complex
  $\F(U_{j_0,\ldots,j_p})$ is an acyclic complex of flat
  $\R(U_{j_0,\ldots,j_p})$-modules whose cycle modules are also
  flat. It follows that the complex $\F(U_{j_0,\ldots,j_p})$ is a
  direct limit
  \begin{equation*}
    \F(U_{j_0,\ldots,j_p}) =
    \varinjlim_{\lambda\in\Lambda} P_\lambda^{U_{j_0,\ldots,j_p}}
  \end{equation*}
  of contractible complexes of projective, hence flat,
  $\R(U_{j_0,\ldots,j_p})$-modules; see for example
  Neeman~\cite[Theorem~8.6]{ANm08}. The functor $i_*$ preserves split
  exact sequences, so
  $i_*(\widetilde{P_\lambda^{U_{j_0,\ldots,j_p}}})$ is for every
  $\lambda\in\Lambda$ a contractible complex of flat sheaves.  The
  functor also preserves direct limits, so $\C^p(\calU,\F)$ is a
  finite direct sum of direct limits of contractible complexes in
  $\catFlatTilde$, hence $\C^p(\calU,\F)$ is itself a direct limit of
  contractible complexes in $\catFlatTilde$. For every complex $\M$ of
  cotorsion sheaves and every $n\in\mathbb{Z}$, Lemma \ref{lemma.stov}
  now yields
  \begin{equation*}
    \H^n\Homab{\C^p(\calU,\F)}{\M} \is
    \H^1\Homab{\C^p(\calU,\F)}{\Sigma^{n-1}\M} = 0
    \quad\text{for}\quad 0 \leq p\leq d\:.
  \end{equation*}
  That is, the complex $\Homab{\mathscr{C}^p(\calU,\F)}{\M}$ is
  acyclic for every $0\leq p\leq d$ and every complex $\M$ of
  cotorsion sheaves. Applying $\Homab{-}{\M}$ to the exact sequence
  \begin{equation*}
    0 \lra \LL_{d-1}\lra\C^{d-1}(\calU,\F) \lra
    \C^d(\calU,\F) \lra 0\:,
  \end{equation*}
  one gets an exact sequence of complexes of abelian groups
  \begin{equation*}
    0 \lra \Homab{\C^d(\calU,\F)}{\M} \lra
    \Homab{\C^{d-1}(\calU,\F)}{\M}
    \lra \Homab{\LL_{d-1}}{\M} \lra 0\:.
  \end{equation*}
  The first two terms are acyclic, and hence so is
  $\Homab{\LL_{d-1}}{\M}$. Repeating this argument $d-1$ more times,
  one concludes that $\Homab{\F}{\M}$ is acyclic, whence one has
  $\Ext{1}{\Ch{\Qcoh{X}}}{\F}{\M} = 0$ per \eqref{eq:11a}.

  The second assertion now follows from \cite[Corollary~3.9]{JGl07}
  which applies as $(\FlatX,\catCot{X})$ is a complete hereditary
  cotorsion pair and $\FlatX$ contains a generator for $\Qcoh{X}$; see
  Remark~\ref{rmk:locfinrep}.
\end{proof}


\section{The stable category of Gorenstein flat-cotorsion sheaves}

\noindent
In this last section, we give a description of the stable category
associated to the cotorsion pair of Gorenstein flat sheaves described
in Theorem~\ref{theorem.gfcotpair}. In particular, we prove Theorems A
and B from the introduction.

Here we use the symbol $\operatorname{hom}$ to denote the morphism
sets in $\Qcoh{X}$ as well as the induced functor to abelian groups.
Further, the tensor product on $\Qcoh{X}$ has a right adjoint functor
denoted ${\mathscr Hom}_{{\rm qc}}$; see for example
\cite[2.1]{DMfSSl11}.

We recall from \cite[Definition 1.1, Proposition 1.3, and Definition
2.1]{CETa}:

\begin{definition}
  \label{dfn:GFC}
  An acyclic complex $\F$ of flat-cotorsion sheaves on $X$ is called
  \emph{totally acyclic} if the complexes $\cathom{\C}{\F}$ and
  $\cathom{\F}{\C}$ are acyclic for every flat-cotorsion sheaf $\C$ on
  $X$.

  A sheaf $\M$ on $X$ is called \emph{Gorenstein flat-cotorsion} if
  there exists a totally acyclic complex $\F$ of flat-cotorsion
  sheaves on $X$ with $\M = \Cy[0]{\F}$. Denote by $\catGFC{X}$ the
  category of Gorenstein flat-cotorsion sheaves on $X$.\footnote{In
    \cite{CETa} this category is denoted $\mathsf{Gor_{FlatCot}}(A)$
    in the case of an affine scheme $X = \Spec{A}$.}
\end{definition}

We proceed to show that the sheaves defined in \ref{dfn:GFC} are
precisely the cotorsion Gorenstein flat sheaves, i.e.\ the sheaves
that are both cotorsion and Gorenstein flat.

The next result is analogous to \cite[Theorem 4.4]{CETa}.

\begin{proposition}
  \label{Ftac_thm}
  Assume that $X$ is semi-separated noetherian. An acyclic complex of
  flat-cotorsion sheaves on $X$ is totally acyclic if and only if it
  is \Ftac.
\end{proposition}

\begin{proof}
  Let $\F$ be a totally acyclic complex of flat-cotorsion sheaves. Let
  $\I$ be an injective sheaf and $\E$ an injective cogenerator in
  $\Qcoh{X}$. By \cite[Lemma 3.2]{DMfSSl11}, the sheaf
  $\Homqc{\I}{\E}$ is flat-cotorsion. The adjunction isomorphism
  \begin{align}
    \label{adj_FC}
    \cathom{\I\otimes \F}{\E}\cong \cathom{\F}{\Homqc{\I}{\E}}
  \end{align}
  along with faithful injectivity of $\E$ implies that $\I\otimes \F$
  is acyclic, hence $\F$ is \Ftac.

  For the converse, let $\F$ be an \Ftac\ complex of flat-cotorsion
  sheaves and $\C$ be a flat-cotorsion sheaf. Recall from
  \cite[Proposition 3.3]{DMfSSl11} that $\C$ is a direct summand of
  $\Homqc{\I}{\E}$ for some injective sheaf $\I$ and injective
  cogenerator $\E$. Thus \eqref{adj_FC} shows that $\cathom{\F}{\C}$
  is acyclic. Moreover, it follows from Theorem~\ref{acyc.cotor} that
  $\Cy[n]{\F}$ is cotorsion for every $n\in\ZZ$, so $\cathom{\C}{\F}$
  is acyclic.
\end{proof}

\begin{theorem}
  \label{equalities}
  Assume that $X$ is semi-separated noetherian. A sheaf on $X$ is
  Gorenstein flat-cotorsion if and only if it is cotorsion and
  Gorenstein flat; that is,
  \begin{equation*}
    \catGFC{X} = \catCot{X}\cap \catGF{X}\:.
  \end{equation*}
\end{theorem}

\begin{proof}
  The containment ``$\subseteq$'' is immediate by
  Theorem~\ref{acyc.cotor} and Proposition~\ref{Ftac_thm}.  For the
  reverse containment, let $\M$ be a cotorsion Gorenstein flat sheaf
  on $X$. There exists an \Ftac\ complex $\F$ of flat sheaves with
  $\M = \Cy[0]{\F}$. As $(\catFlatTilde,\catdgCot{X})$ is a complete
  cotorsion pair, see Remark~\ref{fc}, there is an exact sequence in
  $\Ch{\Qcoh{X}}$
  \begin{equation*}
    0 \lra \F \lra \T \lra \P \lra 0
  \end{equation*}
  with $\T \in \catdgCot{X}$ and $\P \in \catFlatTilde$. As $\F$ and
  $\P$ are \Ftac\, so is $\T$; in particular, $\Cy[n]{\T}$ is
  cotorsion for every $n\in\ZZ$; see Theorem~\ref{acyc.cotor}. The
  argument in \cite[Theorem 5.2]{CETa} now applies verbatim to finish
  the proof.
\end{proof}

\begin{remark}
  \label{stabcats}
  One upshot of Theorem \ref{equalities} is that the Frobenius
  category described in Corollary~\ref{corollary.modelGflat} coincides
  with the one associated to $\catGFC{X}$ per
  \cite[Theorem~2.11]{CETa}. In particular, the associated stable
  categories are equal. One of these is equivalent to the homotopy
  category of the Gorenstein flat model structure and the other is by
  \cite[Corollary 3.9]{CETa} and Proposition~\ref{Ftac_thm} equivalent
  to the homotopy category
  \begin{equation*}
    \catKFtac{\FlatX \cap \catCot{X}}
  \end{equation*}
  of \Ftac\ complexes of flat-cotorsion sheaves on $X$.
\end{remark}

In \cite[2.5]{DMfSSl11} the \emph{pure derived category of flat
  sheaves} on $X$ is the Verdier quotient
\begin{equation*}
  \catD{\FlatX} = \frac{\catK{\FlatX}}{\catKpure{\FlatX}}\:,
\end{equation*}
where $\catKpure{\FlatX}$ is the full subcategory of $\catK{\FlatX}$
of pure acyclic complexes; that is, the objects in $\catKpure{\FlatX}$
are precisely the objects in $\catFlatTilde$. Still following
\cite{DMfSSl11} we denote by $\catDFtac{\FlatX}$ the full subcategory
of $\catD{\FlatX}$ whose objects are \Ftac. As the category
$\catKpure{\FlatX}$ is contained in $\catKFtac{\FlatX}$, it can be
expressed as the Verdier quotient
\begin{equation*}
  \catDFtac{\FlatX} = \frac{\catKFtac{\FlatX}}{\catKpure{\FlatX}}\:.
\end{equation*}

\begin{theorem}
  \label{main_thm}
  Assume that $X$ is semi-separated noetherian.  The composite of
  canonical functors
  \begin{equation*}
    \catKFtac{\FlatX \cap \catCot{X}} \lra  \catKFtac{\FlatX}
    \lra \catDFtac{\FlatX}
  \end{equation*}
  is a triangulated equivalence of categories.
\end{theorem}

\begin{proof}
  In view of Theorem \ref{acyc.cotor} and the fact that
  $(\catFlatTilde,\catdgCot{X})$ is a complete cotorsion pair, see
  Remark~\ref{fc}, the proof of \cite[Theorem 5.6]{CETa} applies
  \emph{mutatis mutandis.}
\end{proof}

We denote by $\stabcatGFC{X}$ the stable category of Gorenstein
flat-cotorsion sheaves; cf.~Remark~\ref{stabcats}. Let $A$ be a
commutative noetherian ring of finite Krull dimension. For the affine
scheme $X = \Spec{A}$ this category is by \cite[Corollary~5.9]{CETa}
equivalent to the stable category $\mathsf{StGPrj}(A)$ of Gorenstein
projective $A$-modules. This, together with the next result, suggests
that $\stabcatGFC{X}$ is a natural non-affine analogue of
$\mathsf{StGPrj}(A)$. Indeed, the category $\catDFtac{\FlatX}$ is
Murfet and Salarian's non-affine analogue of the homotopy category of
totally acyclic complexes of projective modules; see
\cite[Lemma~4.22]{DMfSSl11}.

\begin{corollary}
  \label{cor.main}
  There is a triangulated equivalence of categories
  \begin{equation*}
    \stabcatGFC{X} \simeq \catDFtac{\FlatX}\:.
  \end{equation*}
\end{corollary}

\begin{proof}
  Combine the equivalence
  $\stabcatGFC{X} \simeq \catKFtac{\FlatX \cap \catCot{X}}$ from
  Remark~\ref{stabcats} with Theorem \ref{main_thm}.
\end{proof}

We emphasize that Proposition~\ref{Ftac_thm} and
Theorem~\ref{main_thm} offer another equivalent of the category
$\stabcatGFC{X}$, namely the homotopy category of totally acyclic
complexes of flat-cotorsion sheaves.

A noetherian scheme $X$ is called Gorenstein if the local ring $\Rx$
is Gorenstein for every $x \in X$. We close this paper with a
characterization of Gorenstein schemes in terms of flat-cotorsion
sheaves, it sharpens \cite[Theorem~4.27]{DMfSSl11}.  In a paper in
progress \cite{CDEHLT} we show that Gorensteinness of a scheme $X$ can
be characterized  by the equivalence of the category
$\stabcatGFC{X}$ to a naturally defined singularity category.

\begin{theorem}
  \label{gorX}
  Assume that $X$ is semi-separated noetherian. The following
  conditions are equivalent.
  \begin{itemize}
  \item[$(i)$] $X$ is Gorenstein.
  \item[$(ii)$] Every acyclic complex of flat sheaves on $X$ is \Ftac.
  \item[$(iii)$] Every acyclic complex of flat-cotorsion sheaves on
    $X$ is \Ftac.
  \item[$(iv)$] Every acyclic complex of flat-cotorsion sheaves on $X$
    is totally acyclic.
  \end{itemize}
\end{theorem}

\begin{proof}
  The equivalence of conditions $(i)$ and $(ii)$ is \cite[Theorem
  4.27]{DMfSSl11}, and conditions $(iii)$ and $(iv)$ are equivalent by
  Proposition~\ref{Ftac_thm}. As $(ii)$ evidently implies $(iii)$, it
  suffices to argue the converse.

  Assume that every acyclic complex of flat-cotorsion sheaves is
  \Ftac. Let $\F$ be an acyclic complex of flat sheaves.  As
  $(\catFlatTilde,\catdgCot{X})$ is a complete cotorsion pair, see
  Remark~\ref{fc}, there is an exact sequence in $\Ch{\Qcoh{X}}$,
  \begin{equation*}
    0 \lra \F \lra \C \lra \P \lra 0\:,
  \end{equation*}
  with $\C \in \catdgCot{X}$ and $\P \in \catFlatTilde$. Since $\F$
  and $\P$ are acyclic, the complex $\C$ is also acyclic. Moreover,
  $\C$ is a complex of flat-cotorsion sheaves. By assumption $\C$ is
  \Ftac, and so is $\P$, whence it follows that $\F$ is \Ftac.
\end{proof}

\section*{Acknowledgment}

\noindent
We thank Alexander Sl\'avik for helping us correct a mistake in an
earlier version of the proof of Theorem \ref{thm:equiv_Noetss}. We
also acknowledge the anonymous referee's prompts to improve the
presentation.

  \providecommand{\MR}[1]{\mbox{\href{http://www.ams.org/mathscinet-getitem?mr=#1}{#1}}}
  \renewcommand{\MR}[1]{\mbox{\href{http://www.ams.org/mathscinet-getitem?mr=#1}{#1}}}
  \providecommand{\arxiv}[2][AC]{\mbox{\href{http://arxiv.org/abs/#2}{\sf
  arXiv:#2 [math.#1]}}} \def\cprime{$'$}
\providecommand{\bysame}{\leavevmode\hbox to3em{\hrulefill}\thinspace}
\providecommand{\MR}{\relax\ifhmode\unskip\space\fi MR }
\providecommand{\MRhref}[2]{%
  \href{http://www.ams.org/mathscinet-getitem?mr=#1}{#2}
}
\providecommand{\href}[2]{#2}

\end{document}